\documentclass[12pt,reqno]{amsart}

\usepackage{graphicx,caption}
\usepackage{float}
\usepackage{caption}
\usepackage{tikz,tkz-graph}
\usepackage{hyperref}
\usepackage[usenames,dvipsnames]{pstricks}
\usepackage{epsfig}
\usepackage{pst-grad} 
\usepackage{pst-plot} 

\newcommand{\reg}{\operatorname{reg}}
\newcommand{\iv}{\operatorname{iv}}
\newcommand{\degr}{\operatorname{deg}}

\newcommand{\tor}{\operatorname{Tor}}
\newcommand{\leaves}{\operatorname{L}}
\newcommand{\J}{\mathcal{J}}

\newtheorem{theorem}{Theorem}[section]

\newtheorem{corollary}[theorem]{Corollary}

\newtheorem{example}[theorem]{Example}

\newtheorem{fact}[theorem]{Fact}

\begin{document}

\title{Regularity of Binomial Edge Ideals of Certain Block Graphs}
\author{A. V. Jayanthan}
\address{Department of Mathematics, Indian Institute of Technology
Madras, Chennai, INDIA - 600036.}
\email{jayanav@iitm.ac.in}
\author{N. Narayanan}
\address{Department of Mathematics, Indian Institute of Technology
Madras, Chennai, INDIA - 600036.}
\email{naru@iitm.ac.in}
\author{B. V. Raghavendra Rao}
\address{Department of Computer Science and Engineering, Indian
Institute of Technology Madras, Chennai, INDIA - 600036.}
\email{bvrr@iitm.ac.in}

\maketitle

\begin{abstract}
We prove that the regularity of binomial edge ideals of graphs
obtained by gluing two graphs at a free vertex is the sum of the
regularity of individual graphs. As a consequence, we generalize
certain results of Zafar and Zahid.  We obtain an improved lower bound
for the regularity of trees. Further, we characterize trees which
attain the lower bound. We prove an upper bound for the regularity of
certain subclass of block-graphs. As a consequence we obtain sharp
upper and lower bounds for a class of trees called lobsters. 
\end{abstract}

\section{Introduction}
Let $G$ be a simple graph on the vertex set $[n]$. Let $S = K[x_1,
\ldots, x_n,y_1,\ldots,y_n]$ be the polynomial ring in $2n$
variables, where $K$ is a field. Then the ideal $J_G$ generated by
$\{x_iy_j-x_jy_i \mid (i,j) \mbox{ is an edge in } G\}$ is called
the binomial edge ideal of $G$. This was introduced by Herzog et
al., \cite{hhhkr10} and independently by Ohtani, \cite{ohtani11}.
Recently, there have been many results relating the combinatorial data
of graphs with the algebraic properties of the corresponding binomial
edge ideals, see \cite{cdi14}, \cite{cr2011}, \cite{ehh11},
\cite{mm2013}, \cite{mk2012}, \cite{mk2013},  \cite{zz13}. In
particular, there have been active research connecting algebraic
invariants of the binomial edge ideals such as Castelnuovo-Mumford
regularity, depth, Betti numbers etc., with combinatorial invariants
associated with graphs such as length of maximal induced path,
number of maximal cliques, matching number. For example, Matsuda and
Murai proved that $\ell \leq \reg(S/J_G) \leq n-1$, where $\ell$ is
the length of the longest induced path in $G$, \cite{mm2013}. They
conjectured that if $\reg(S/J_G) = n-1$, then $G$ is a path of length
$n$. In \cite{km16}, Kiani and Saeedi Madani proved the conjecture.
Chaudhry et al. proved that if $T$ is a tree whose longest induced
path has length $\ell$, then $\reg(S/J_T) = \ell$ if and only if $T$
is a caterpillar, \cite{cdi14}. Therefore, the trees that attain the
minimal or maximal regularity have been characterized. However, for
most of the graph classes, the Matsuda-Murai bounds are far from being
tight. Saeedi Madani and Kiani, \cite{mk2012}, proved that if $G$ is a
closed graph, then $\reg(S/J_G) \leq c(G)$, where $c(G)$ is the number
of maximal cliques in $G$. Here, a graph is said to be closed if its
binomial edge ideal has a quadratic Gr\"obner basis.  They generalized
this result to the case of binomial edge ideal of a pair of a closed
graph and a complete graph, and proposed conjectured that for any
graph $G$, $\reg(S/J_G) \leq c(G)$, \cite{mk2013}. In
\cite{mk-arxiv-13}, they proved the conjecture for generalized block
graph. In \cite{ez2015}, Ene and Zarojanu proved that if $G$ is a
chordal graph with the property that any two distinct maximal cliques
intersect in at most one vertex, then $\reg(S/J_G) \leq c(G)$. 

Though the bound obtained for block graphs by Madani and Kiani is
sharp, there are several subclasses of block graphs, including trees,
where the upper bound is more than the actual regularity (for example,
caterpillar, \cite{cdi14}). In this article, we study the regularity
of binomial edge ideals of certain classes of block graphs, and in
particular trees.

In \cite{rr14}, Rauf and Rinaldo studied binomial edge ideals of
graphs obtained by gluing two graphs at free vertices. We extend their
arguments to observe that the regularity of the binomial edge ideal of
a graph obtained by gluing two graphs at free vertices is equal to
the sum of the regularities the binomial edge ideals of the
individual graphs, Theorem \ref{glue}. As a consequence, we obtain
precise expressions for the regularities of several classes of trees and
block graphs, (Corollaries \ref{reg-sum}, \ref{block-sum} and
\ref{g3t3}).

The lower bound for the regularity of a binomial edge ideal given by
Matsuda and Murai, namely, the length of the longest induced path,
\cite{mm2013}, is the best lower bound known as of now.  By using
inductive application of Theorem \ref{glue}, we obtain a lower bound
for the regularity of binomial edge ideals of trees in terms of the
number of internal vertices, Theorem \ref{tree-regularity}. We
characterize trees which attain the lower bound in terms of presence
of a specific tree as a subgraph, Theorem \ref{minreg}.

We then move on to study certain subclasses of block graphs and obtain
improved upper bounds for their regularity, Theorems \ref{gmnw} and
\ref{clique-whisk}. As a consequence
we get an upper bound for the
regularity of the binomial edge ideals of lobsters (see Section 2 for
definition), Corollary \ref{cor:lobsternowhisker}. 
We also obtain a precise
expression for the regularity of binomial edge ideals of a subclass of
lobsters, called pure lobsters, in Corollary \ref{purelobster}. 

\vskip 2mm

\section{Preliminaries} 

In this section, we set up the basic definitions and notation.

Let $G$ be a finite simple graph. A vertex $x$ of $G$ is said to be
a cut vertex if $G \setminus \{x\}$ has strictly more connected
components than $G$. A block of $G$ is a maximal subgraph
without a cut vertex. A graph $G$ is a \textit{block graph} if every
block of $G$ is a complete graph.

Let $T$ be a tree and $\leaves(T)=\{v\in V(T)|\degr(v)=1\}$ be the set of all
leaves of $T$. We say that a tree $T$ is a {\em caterpillar} if $T \setminus
\leaves(T)$ is either empty or is a simple path. 
Similarly, a tree  $T$ is said to be a {\em lobster}, if $T \setminus
\leaves(T)$  is a caterpillar, \cite{g1994}. Observe that every caterpillar is also
a lobster.  A longest path in a lobster is called a {\em spine} of the
lobster. Note that  given any spine, every edge of a
caterpillar is incident to it.  With respect to a fixed spine $P$, the
pendant edges incident with $P$ are called {\em whiskers}. It can be
seen that every non-leaf vertex $u$ not incident on a fixed spine
$P$ of a lobster forms the center of a star ($K_{1,m}, m \ge 2$). 
Each such star is said to be a  {\em limb} with respect to $P$. 
More generally, given a vertex $v$ on any simple path $P$, we can
attach a star  $(K_{1,m}, m\ge 2)$  with center $u$  by identifying
exactly one of the leaves of the star with  $v$. Such a star is called
{\em  a limb attached to} $P$. 

Note that the limbs and whiskers depend on the spine. Whenever a spine
is fixed, we will refer to them simply as limb and whisker.

%

\begin{example} Let $G$ denote the given graph on $10$ vertices:

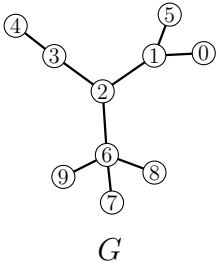
\begin{minipage}{\linewidth}
\begin{minipage}{0.2\linewidth}
\captionsetup[figure]{labelformat=empty}
\begin{figure}[H]

  \begin{tikzpicture}[scale=0.5]
  \tikzset{VertexStyle/.style = {shape = circle, inner sep = 1pt, outer sep = 0pt, minimum size = 0 pt, scale=0.7}}
\useasboundingbox (0,0) rectangle (5.0cm,5.0cm);
\definecolor{cv0}{rgb}{0.0,0.0,0.0}
\definecolor{cfv0}{rgb}{1.0,1.0,1.0}
\definecolor{clv0}{rgb}{0.0,0.0,0.0}
\definecolor{cv1}{rgb}{0.0,0.0,0.0}
\definecolor{cfv1}{rgb}{1.0,1.0,1.0}
\definecolor{clv1}{rgb}{0.0,0.0,0.0}
\definecolor{cv2}{rgb}{0.0,0.0,0.0}
\definecolor{cfv2}{rgb}{1.0,1.0,1.0}
\definecolor{clv2}{rgb}{0.0,0.0,0.0}
\definecolor{cv3}{rgb}{0.0,0.0,0.0}
\definecolor{cfv3}{rgb}{1.0,1.0,1.0}
\definecolor{clv3}{rgb}{0.0,0.0,0.0}
\definecolor{cv4}{rgb}{0.0,0.0,0.0}
\definecolor{cfv4}{rgb}{1.0,1.0,1.0}
\definecolor{clv4}{rgb}{0.0,0.0,0.0}
\definecolor{cv5}{rgb}{0.0,0.0,0.0}
\definecolor{cfv5}{rgb}{1.0,1.0,1.0}
\definecolor{clv5}{rgb}{0.0,0.0,0.0}
\definecolor{cv6}{rgb}{0.0,0.0,0.0}
\definecolor{cfv6}{rgb}{1.0,1.0,1.0}
\definecolor{clv6}{rgb}{0.0,0.0,0.0}
\definecolor{cv7}{rgb}{0.0,0.0,0.0}
\definecolor{cfv7}{rgb}{1.0,1.0,1.0}
\definecolor{clv7}{rgb}{0.0,0.0,0.0}
\definecolor{cv8}{rgb}{0.0,0.0,0.0}
\definecolor{cfv8}{rgb}{1.0,1.0,1.0}
\definecolor{clv8}{rgb}{0.0,0.0,0.0}
\definecolor{cv9}{rgb}{0.0,0.0,0.0}
\definecolor{cfv9}{rgb}{1.0,1.0,1.0}
\definecolor{clv9}{rgb}{0.0,0.0,0.0}
\definecolor{cv0v1}{rgb}{0.0,0.0,0.0}
\definecolor{cv1v2}{rgb}{0.0,0.0,0.0}
\definecolor{cv1v5}{rgb}{0.0,0.0,0.0}
\definecolor{cv2v3}{rgb}{0.0,0.0,0.0}
\definecolor{cv2v6}{rgb}{0.0,0.0,0.0}
\definecolor{cv3v4}{rgb}{0.0,0.0,0.0}
\definecolor{cv6v7}{rgb}{0.0,0.0,0.0}
\definecolor{cv6v8}{rgb}{0.0,0.0,0.0}
\definecolor{cv6v9}{rgb}{0.0,0.0,0.0}
\Vertex[style={minimum size=1.0cm,draw=cv0,fill=cfv0,text=clv0,shape=circle},LabelOut=false,L=\hbox{$0$},x=5.0cm,y=3.9748cm]{v0}
\Vertex[style={minimum size=1.0cm,draw=cv1,fill=cfv1,text=clv1,shape=circle},LabelOut=false,L=\hbox{$1$},x=3.6851cm,y=3.9089cm]{v1}
\Vertex[style={minimum size=1.0cm,draw=cv2,fill=cfv2,text=clv2,shape=circle},LabelOut=false,L=\hbox{$2$},x=2.313cm,y=2.9586cm]{v2}
\Vertex[style={minimum size=1.0cm,draw=cv3,fill=cfv3,text=clv3,shape=circle},LabelOut=false,L=\hbox{$3$},x=1.028cm,y=3.8972cm]{v3}
\Vertex[style={minimum size=1.0cm,draw=cv4,fill=cfv4,text=clv4,shape=circle},LabelOut=false,L=\hbox{$4$},x=0.0cm,y=4.7034cm]{v4}
\Vertex[style={minimum size=1.0cm,draw=cv5,fill=cfv5,text=clv5,shape=circle},LabelOut=false,L=\hbox{$5$},x=4.096cm,y=5.0cm]{v5}
\Vertex[style={minimum size=1.0cm,draw=cv6,fill=cfv6,text=clv6,shape=circle},LabelOut=false,L=\hbox{$6$},x=2.4251cm,y=1.2666cm]{v6}
\Vertex[style={minimum size=1.0cm,draw=cv7,fill=cfv7,text=clv7,shape=circle},LabelOut=false,L=\hbox{$7$},x=2.5622cm,y=0.0cm]{v7}
\Vertex[style={minimum size=1.0cm,draw=cv8,fill=cfv8,text=clv8,shape=circle},LabelOut=false,L=\hbox{$8$},x=3.6902cm,y=0.7971cm]{v8}
\Vertex[style={minimum size=1.0cm,draw=cv9,fill=cfv9,text=clv9,shape=circle},LabelOut=false,L=\hbox{$9$},x=1.2685cm,y=0.6822cm]{v9}
\Edge[lw=0.03cm,style={color=cv0v1,},](v0)(v1)
\Edge[lw=0.03cm,style={color=cv1v2,},](v1)(v2)
\Edge[lw=0.03cm,style={color=cv1v5,},](v1)(v5)
\Edge[lw=0.03cm,style={color=cv2v3,},](v2)(v3)
\Edge[lw=0.03cm,style={color=cv2v6,},](v2)(v6)
\Edge[lw=0.03cm,style={color=cv3v4,},](v3)(v4)
\Edge[lw=0.03cm,style={color=cv6v7,},](v6)(v7)
\Edge[lw=0.03cm,style={color=cv6v8,},](v6)(v8)
\Edge[lw=0.03cm,style={color=cv6v9,},](v6)(v9)
\end{tikzpicture}

%
\caption{$G$}
\end{figure}
\end{minipage}
\begin{minipage}{0.75\linewidth}
In this example, $G$ has many longest induced paths. The path
induced by the vertices $\{0,1,2,3,4\}, ~ \{0,1,2,6,9\}$ are two
such (there are more) paths. Let $P$ denote the path induced by the
vertices $\{0,1,2,3,4\}$. Then $(1,5)$ is a whisker with respect to
$P$. Also the subgraph induced by the vertices $\{2,6,7,8,9\}$ is a
limb with respect to $P$. If we consider $\{0,1,2,6,9\}$ as spine $P$,
then $\{(1,5), (6,7),(6,8)\}$ are whiskers with respect to $P$ and the
path induced by $\{2,3,4\}$ is a limb.
\end{minipage}
\end{minipage}
\end{example}

We now describe the construction of a useful exact sequence introduced
by Ene, Herzog and Hibi, \cite{ehh11}.
\subsection{Ene-Herzog-Hibi Process}\label{ehh-process}
Let $G$ be a block graph, $\Delta(G)$ be the clique complex of $G$ and
$F_1, \ldots, F_r$ be a leaf order on the facets of $\Delta(G)$.
Assume that $r > 1$. Let $v \in V(G)$ be the unique vertex in $F_r$
such that $F_r \cap F_j \subseteq \{v\}$ for all $j < r$. Let $G'$ be the graph
obtained by adding necessary edges to $G$ so that $N(v) \cup \{v\}$ is
a clique. Let $G''$ be the graph induced on $G \setminus \{v\}$ and
$H$ be the graph induced on $G' \setminus \{v\}$. Then there exists an
exact sequence
\begin{eqnarray}\label{ehh-es}
  0 \to S/J_G \to S/J_{G'} \oplus S/J_{G''} \to S/J_H \to 0.
\end{eqnarray}
We call $G', ~G''$ and $H$ to be the graphs obtained by applying
EHH-process on $G$ with respect to $v$. This exact sequence has been
found extremely useful in inductive arguments in the study of
homological properties of the binomial edge ideals.

\section{Regularity via gluing}
In this section, we describe the process of gluing and use it to
obtain precise regularity expressions for certain classes of graphs.
Let $G$ be a graph and $v$ be a cut vertex in $G$. Let
$G_1,\ldots,G_k$ be the components of $G \setminus \{v\}$ and $G_i' =
G[V(G_i)\cup\{v\}]$,
the subgraph of $G$ induced by $V(G_i)\cup\{v\}$.
Then, $G_1', \ldots, G_k'$ is called the {\em split} of $G$ at $v$ and  we say that $G$ is obtained by
{\em gluing} $G_1, \ldots, G_k$ at $v$.

\begin{theorem}\label{glue}

Let $G_1$ and $G_2$ be the split of a graph $G$ at $v$. If  $v$ is a free vertex in
both $G_1$ and $G_2$, then 
\[\reg(S/J_G) = \reg(S/J_{G_1}) + \reg(S/J_{G_2}).\]
\end{theorem}

\begin{proof}
Let $G_1$ and $G_2$ be graphs on the vertices $\{1,\ldots, n\}$ and
$\{n+1,\ldots, n+m\}$ respectively. Assume that $n$ is a free vertex
in $G_1$ and $n+m$ is a free vertex in $G_2$. Let $G$ be the graph
obtained by identifying vertices $n$ and $n+m$ in $G_1 \cup G_2$, i.e., $v=n=n+m$.
Let $G' = G_1 \cup G_2$ and $S' = K[x_1,\ldots,x_{n+m}, y_1,\ldots,
y_{n+m}]$. 
Then
it can be easily seen that $S/J_G \cong S'/(J_{G'}+(x_n-x_{n+m},
y_n-y_{n+m}))$. From the proof of Theorem 2.7 in \cite{rr14}, it
follows that $(x_n-x_{n+m}, y_n-y_{n+m})$ is a regular sequence on
$S'/J_{G'}$. Hence the assertion follows.
\end{proof}

As an immediate consequence, we have the following:

\vskip 2mm\noindent
\begin{corollary}\label{reg-sum}
Let $G = G_1 \cup \cdots \cup G_k$ be such that
	\begin{enumerate}
	  \item for $i\neq j,$ if $G_i \cap G_j \neq \emptyset$, then $G_i \cap G_j =
		\{v_{ij}\},$ for some vertex $v_{ij}$ which is a free vertex in $G_i$ as
		well as $G_j$;
	  \item for distinct $i, j, k$, $G_i \cap G_j \cap G_k =
		\emptyset$.
	\end{enumerate}
Then $\reg S/J_G = \sum_{i=1}^k \reg S/J_{G_i}.$
\end{corollary}

Recall that 
for a (generalized) block graph $G$, $\reg S/J_G \leq c(G)$,
\cite{mk-arxiv-13}. We obtain a subclass of block graphs which attain
this bound.

\begin{corollary}\label{block-sum}
If $G$ is a block graph such that no vertex is contained in more than
two maximal cliques, then $\reg S/J_G = c(G)$. 
\end{corollary}
\begin{proof}
We use induction on $c(G)$. If $c(G) = 1$, then $G$ is a complete
graph and hence $\reg S/J_G = 1$. Now assume that $c(G) > 1$.
Consider any cut vertex $v$
of $G$. Let $G_1$ and $G_2$  be the split of $G$ at $\{v\}$.
Then, $c(G)=c(G_1)+c(G_2)$.
Now the result follows from Corollary \ref{reg-sum} and induction
hypothesis.  
\end{proof}
In \cite{zz13}, Zafar and Zahid considered special classes of graphs
called $\mathcal{G}_3$ and $\mathcal{T}_3$ and obtained the
regularities of the corresponding binomial edge ideals. We generalize their results:
\begin{corollary}\label{g3t3}
\begin{enumerate}
  \item Let $P_1, \ldots, P_s$ be paths of lengths $r_1, \ldots, r_s$
	respectively. Let $G$ be the graph obtained by identifying a leaf
	of $P_i$ with the $i$-th vertex of the complete graph $K_s$. Then $\reg(S/J_G) =
	1+\sum_{i=1}^s r_i$. 

  \item Let $P_1, \ldots, P_k$ be paths of lengths $r_1, \ldots, r_k$
	respectively. Let $G$ be the graph obtained by identifying a leaf
	of $P_i$ with the $i$-th leaf of the star $K_{1,k}$. Then, $\reg(S/J_G) = 2
	+ \sum_{i=1}^k r_i$. 
\end{enumerate}
\end{corollary}
\begin{proof}
  Both the assertions follow from Theorem \ref{glue}.
\end{proof}

\section{Regularity of Block Graphs}
In this section, we study the regularity of binomial edge ideals of
certain block graphs, and in particular trees. We first obtain a lower bound for the regularity. We then
consider a class of graphs, called lobsters, which are a
generalization of caterpillars. We generalize a result of Chaudhry et
al. to obtain sharp upper bounds for the regularity of binomial edge
ideals of lobsters.
It was shown by Matsuda and Murai, \cite[Corollary 2.3]{mm2013},
that for any graph $G$, $\reg(S/J_G) \geq \ell$, where $\ell$ is the
length of the longest induced path in $G$. Below we prove a much
improved lower bound, for the class of trees.

For a tree $T$, let
$\iv(T):=\#\{\mbox{internal vertices of } T\}$.  Given a tree $T$, it is easy to
see that one can construct $T$ from the trivial graph by adding vertices $v_i$
to $T_{i-1}$ at step $i$ to get  $T_i$   so that $v_i$ is a leaf  in the tree
$T_i$. Any such ordering of vertices is called a {\em leaf ordering}.  

\begin{theorem}\label{tree-regularity}
For a tree $T$, $\reg(S/J_T) \geq \iv(T)+1$.
\end{theorem}

\begin{proof} Let $v_1, \ldots, v_r$ be a leaf ordering of the
vertices of $G$, and let $G_i$ be the subgraph of $G$  induced by $v_1,
\ldots, v_i$. Let $m_i$ denote the number of internal vertices of
$G_i$. We argue by induction on $i$. If $i = 2$, then $G_2$
is an edge and hence $\reg(S/J_{G_2}) = 1$. Therefore, the result holds.
Assume the result for $G_i$. Then $G_{i+1}$ is obtained by adding a
leaf $v_{i+1}$ to some vertex $v$ of $G_i$. If $v$ is a leaf in
$G_i$, then $v$ is a free vertex in $G_i$, and hence by Theorem
\ref{glue},  $\reg(S/J_{G_{i+1}})= \reg(S/J_{G_i})+1$.
Further, $v$ becomes a new internal vertex in $G_{i+1}$, i.e.,
$m_{i+1}=m_i +1$, and therefore the result holds.
If $v$ is an internal vertex in $G_i$, then $m_{i+1}=m_i$ and since
$G_i$ is an induced subgraph of $G_{i+1}$,  $\reg(S/J_{G_{i+1}}) \ge
\reg(S/J_{G_i}) \geq m_i + 1 = m_{i+1}+1$ as required.
\end{proof}

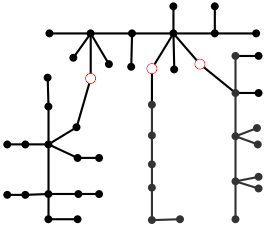
\begin{minipage}{\linewidth}
\begin{minipage}{0.40\linewidth}
\newrgbcolor{tttttt}{0.2 0.2 0.2}
\psset{xunit=1.1cm,yunit=1.2cm,algebraic=true,dotstyle=o,dotsize=3pt 0,linewidth=0.8pt,arrowsize=3pt 2,arrowinset=0.25}
\begin{pspicture*}(3.37,-4.9)(8.02,-2.14)
\psline(5,-2.5)(5.5,-2.5)
\psline(6,-2.5)(5.5,-2.5)
\psline(6,-2.5)(6.5,-2.5)
\psline(6.5,-2.5)(7,-2.5)
\psline(7,-2.5)(7.5,-2.5)
\psline(5.5,-2.5)(5.29,-2.77)
\psline(5.5,-2.5)(5.72,-2.84)
\psline(6,-2.5)(5.99,-2.87)
\psline(6.5,-2.5)(6.51,-2.9)
\psline(6.5,-2.5)(6.5,-2.2)
\psline(7,-2.5)(7,-2.2)
\psline(4.99,-3.31)(4.99,-3.73)
\psline(4.99,-3.73)(4.99,-4.28)
\psline(4.99,-4.28)(4.99,-4.56)
\psline(4.99,-3.73)(5.33,-3.54)
\psline(4.99,-3.73)(5.34,-3.88)
\psline(5.34,-3.88)(5.6,-3.88)
\psline(4.99,-3.73)(4.71,-3.73)
\psline(4.71,-3.73)(4.49,-3.73)
\psline(4.99,-3.31)(4.98,-2.99)
\psline(4.99,-4.56)(5.34,-4.56)
\psline(4.99,-4.28)(4.71,-4.29)
\psline(4.71,-4.29)(4.49,-4.29)
\psline(4.99,-4.28)(5.35,-4.28)
\psline(5.35,-4.28)(5.6,-4.28)
\psline(5.5,-2.5)(5.5,-3)
\psline(5.5,-3)(5.33,-3.54)
\psline[linecolor=tttttt](7.25,-2.75)(7.25,-3.16)
\psline[linecolor=tttttt](7.25,-3.16)(7.25,-3.64)
\psline[linecolor=tttttt](7.25,-3.64)(7.25,-4.14)
\psline[linecolor=tttttt](7.25,-4.14)(7.25,-4.56)
\psline[linecolor=tttttt](7.25,-3.64)(7.51,-3.55)
\psline[linecolor=tttttt](7.25,-3.64)(7.52,-3.73)
\psline[linecolor=tttttt](7.25,-4.14)(7.53,-4.09)
\psline[linecolor=tttttt](7.25,-4.14)(7.53,-4.22)
\psline(6.5,-2.5)(6.82,-2.84)
\psline(6.24,-2.89)(6.24,-3.29)
\psline[linecolor=tttttt](6.24,-3.29)(6.24,-3.63)
\psline[linecolor=tttttt](6.24,-3.63)(6.24,-3.95)
\psline[linecolor=tttttt](6.24,-3.95)(6.24,-4.21)
\psline[linecolor=tttttt](6.24,-4.21)(6.24,-4.57)
\psline[linecolor=tttttt](6.24,-4.57)(6.58,-4.56)
\psline(6.5,-2.5)(6.24,-2.89)
\psline(6.82,-2.84)(7.25,-3.16)
\psline(7.25,-2.75)(7.53,-2.75)
\psline(7.25,-3.16)(7.53,-3.16)
\begin{scriptsize}
\psdots[dotstyle=*](5,-2.5)
\rput[bl](5.01,-2.49){}
\psdots[dotstyle=*](5.5,-2.5)
\rput[bl](5.51,-2.49){}
\psdots[dotstyle=*](6,-2.5)
\rput[bl](6.01,-2.49){}
\psdots[dotstyle=*](6.5,-2.5)
\rput[bl](6.51,-2.49){}
\psdots[dotstyle=*](7,-2.5)
\rput[bl](7.01,-2.49){}
\psdots[dotstyle=*](7.5,-2.5)
\rput[bl](7.51,-2.49){}
\psdots[dotstyle=*](5.29,-2.77)
\rput[bl](5.3,-2.76){}
\psdots[dotstyle=*](5.72,-2.84)
\rput[bl](5.73,-2.83){}
\psdots[dotstyle=*](5.99,-2.87)
\rput[bl](6,-2.86){}
\psdots[dotstyle=*](6.51,-2.9)
\rput[bl](6.51,-2.88){}
\psdots[dotstyle=*](6.5,-2.2)
\rput[bl](6.5,-2.18){}
\psdots[dotstyle=*](7,-2.2)
\rput[bl](7.01,-2.19){}
\psdots[dotstyle=*](4.99,-3.31)
\rput[bl](5,-3.3){}
\psdots[dotstyle=*](4.99,-3.73)
\rput[bl](5,-3.72){}
\psdots[dotstyle=*](4.99,-4.28)
\rput[bl](5,-4.27){}
\psdots[dotstyle=*](4.99,-4.56)
\rput[bl](5,-4.55){}
\psdots[dotstyle=*](5.33,-3.54)
\rput[bl](5.34,-3.53){}
\psdots[dotstyle=*](5.34,-3.88)
\rput[bl](5.34,-3.87){}
\psdots[dotstyle=*](5.6,-3.88)
\rput[bl](5.61,-3.87){}
\psdots[dotstyle=*](4.71,-3.73)
\rput[bl](4.72,-3.72){}
\psdots[dotstyle=*](4.49,-3.73)
\rput[bl](4.5,-3.72){}
\psdots[dotstyle=*](4.98,-2.99)
\rput[bl](4.99,-2.97){}
\psdots[dotstyle=*](5.34,-4.56)
\rput[bl](5.35,-4.54){}
\psdots[dotstyle=*](4.71,-4.29)
\rput[bl](4.72,-4.27){}
\psdots[dotstyle=*](4.49,-4.29)
\rput[bl](4.49,-4.28){}
\psdots[dotstyle=*](5.35,-4.28)
\rput[bl](5.35,-4.27){}
\psdots[dotstyle=*](5.6,-4.28)
\rput[bl](5.61,-4.27){}
\psdots[dotsize=4pt 0,linecolor=red](5.5,-3)
\rput[bl](5.51,-2.97){}
\psdots[dotstyle=*,linecolor=tttttt](7.25,-2.75)
\rput[bl](7.26,-2.74){}
\psdots[dotstyle=*,linecolor=tttttt](7.25,-3.16)
\rput[bl](7.26,-3.15){}
\psdots[dotstyle=*,linecolor=tttttt](7.25,-3.64)
\rput[bl](7.26,-3.63){}
\psdots[dotstyle=*,linecolor=tttttt](7.25,-4.14)
\rput[bl](7.26,-4.13){}
\psdots[dotstyle=*,linecolor=tttttt](7.25,-4.56)
\rput[bl](7.26,-4.55){}
\psdots[dotstyle=*,linecolor=tttttt](7.51,-3.55)
\rput[bl](7.52,-3.54){}
\psdots[dotstyle=*,linecolor=tttttt](7.52,-3.73)
\rput[bl](7.53,-3.72){}
\psdots[dotstyle=*,linecolor=tttttt](7.53,-4.09)
\rput[bl](7.54,-4.08){}
\psdots[dotstyle=*,linecolor=tttttt](7.53,-4.22)
\rput[bl](7.54,-4.21){}
\psdots[dotsize=4pt 0,linecolor=red](6.82,-2.84)
\rput[bl](6.83,-2.82){}
\psdots[dotsize=4pt 0,linecolor=red](6.24,-2.89)
\rput[bl](6.25,-2.87){}
\psdots[dotstyle=*,linecolor=tttttt](6.24,-3.29)
\rput[bl](6.25,-3.28){}
\psdots[dotstyle=*,linecolor=tttttt](6.24,-3.63)
\rput[bl](6.25,-3.62){}
\psdots[dotstyle=*,linecolor=tttttt](6.24,-3.95)
\rput[bl](6.25,-3.94){}
\psdots[dotstyle=*,linecolor=tttttt](6.24,-4.21)
\rput[bl](6.25,-4.2){}
\psdots[dotstyle=*,linecolor=tttttt](6.24,-4.57)
\rput[bl](6.25,-4.55){}
\psdots[dotstyle=*,linecolor=tttttt](6.58,-4.56)
\rput[bl](6.59,-4.55){}
\psdots[dotstyle=*](7.53,-2.75)
\rput[bl](7.53,-2.73){}
\psdots[dotstyle=*](7.53,-3.16)
\rput[bl](7.53,-3.14){}
\end{scriptsize}
\end{pspicture*}
\end{minipage}
\begin{minipage}{0.5\linewidth}
Let $T$ be the tree given on the left. It follows from \cite[Theorem
4.1]{cdi14} and Theorem \ref{glue} that $\reg(S/J_T) = 26 = \iv(T)+1$,
while the longest path of $T$ has length $15$.
\end{minipage}
\end{minipage}

\vskip 2mm
\noindent
\begin{minipage}{\linewidth}
  \begin{minipage}{0.75\linewidth}
	It is interesting to note that the graph $\mathcal{J}$, which
we call {\em Jewel}, is the smallest tree for which
$\reg(S/J_\mathcal{J})>
\iv(\mathcal{J}) +1$. In fact, we can
make the gap between the regularity and the number of internal
vertices arbitrarily large by attaching  edge disjoint copies of
Jewel to leaves of any arbitrary tree. For example,
Figure~\ref{fig:doublejewel}, which is two copies of the jewel
superimposed together,  has regularity 12, much larger than the number
of internal vertices which is 7.
\end{minipage}
\begin{minipage}{0.2\linewidth}
\captionsetup[figure]{labelformat=empty}
\begin{figure}[H]
\begin{tikzpicture}[line cap=round,line join=round,>=triangle 45,x=0.6cm,y=0.6cm]
\draw (2.82,-2.52)-- (2.82,-3.86);
\draw (2.82,-2.52)-- (3.86,-2.02);
\draw (2.82,-2.52)-- (1.78,-2.04);
\draw (3.86,-2.02)-- (4.82,-2.52);
\draw (3.86,-2.02)-- (3.84,-1.1);
\draw (2.82,-3.86)-- (3.82,-4.52);
\draw (2.82,-3.86)-- (1.82,-4.52);
\draw (1.78,-2.04)-- (0.82,-2.52);
\draw (1.78,-2.04)-- (1.76,-1.02);
\begin{scriptsize}
\fill [color=black] (2.82,-2.52) circle (2.5pt);
\fill [color=black] (3.86,-2.02) circle (2.5pt);
\fill [color=black] (2.82,-3.86) circle (2.5pt);
\fill [color=black] (1.78,-2.04) circle (2.5pt);
\fill [color=black] (4.82,-2.52) circle (2.5pt);
\fill [color=black] (3.84,-1.1) circle (2.5pt);
\fill [color=black] (3.82,-4.52) circle (2.5pt);
\fill [color=black] (1.82,-4.52) circle (2.5pt);
\fill [color=black] (0.82,-2.52) circle (2.5pt);
\fill [color=black] (1.76,-1.02) circle (2.5pt);
\end{scriptsize}
\end{tikzpicture}
\caption{$\mathcal{J}$: Jewel}
\end{figure}
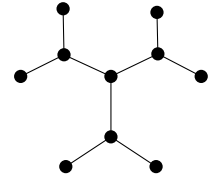
\end{minipage}
\end{minipage}
We now characterize trees which attain the minimal regularity.
\begin{theorem}\label{minreg}
A tree $T$ contains  Jewel as a subgraph if and only if $\reg(S/J_T)
 \geq \iv(T)+2$.
\end{theorem}

\begin{proof}
Suppose $T$ is a tree on $[n]$ containing \textit{Jewel}, $\J$, as a
subgraph. Note that there is a leaf ordering $v_1, \ldots, v_n$ such
that $V(\J) = \{v_1, \ldots, v_{10}\}$. Recall that $\reg(S/J_\J) = 6 =
\iv(\J) + 2$. Let $G_i$ denote the subgraph of $T$ on the vetex set $\{v_1,
\ldots, v_i\}, i \geq 10$. For each $i \geq 10$, $\reg(S/J_{G_{i+1}}) = \reg(S/J_{G_i}) + 1$ if
the neighbor of $v_{i+1}$ is a leaf in $G_i$ and $\reg(S/J_{G_{i+1}})
\geq \reg(S/J_{G_i})$ otherwise. Note also that  the neighbor of $v_{i+1}$ is a leaf in $G_i$ if
and only if $\iv(G_{i+1}) = \iv(G_i) + 1$. Since $\reg(S/J_{G_{10}}) = \iv(G_{10}) +
2$, we get that $\reg(S/J_{G_i}) \geq \iv(G_i) + 2$ for all $i \geq
10$. Hence the assertion follows. 

Conversely, suppose $\reg(S/J_T) \geq \iv(T)+2$. First assume that
$T$ does not have a vertex of degree $2$.  If $T$ is a caterpillar, then
by~\cite{cdi14}, $\reg(S/J_T) = \iv(T)+1$. Therefore, $T$ is not a caterpillar.
Then it contains the $Y$ graph ($K_{1,3}$ attached with a leaf at each of its
leaf vertices) as a subgraph~\cite[Theorem 2.2.19]{west}. Since $T$ does not have vertices of
degree $2$, each degree 2 vertex in the $Y$ graph must have one more neighbour
in $T$, which induces a Jewel in $T$.

Now assume that $T$ contains a vertex of degree 2.  Let $T'$ be a
minimal (with respect to the number of vertices) subgraph of $T$ so
that $\reg(S/J_{T'}) \ge \iv(T')+2$. If $T'$ has no vertex of degree
$2$, then $T'$ contains a jewel. Suppose $T'$ contains degree $2$
vertex, say $v$. Let $T_1$ and $T_2$ be the split of $T'$ at $\{v\}$.
Note that $\iv(T') = \iv(T_1) + \iv(T_2) + 1$. By Theorem \ref{glue},
$\reg(S/J_{T'}) = \reg(S/J_{T_1}) + \reg(S/J_{T_2}) \geq \iv(T_1) +
\iv(T_2) + 3$. Hence there exists $i \in \{1,2\}$ such that
$\reg(S/J_{T_i}) \geq \iv(T_i)+2$. Since $T_i$ is a subgraph of $T'$,
this contradicts the minimality of $T'$. Hence $T'$ does not contain a
degree $2$ vertex. Therefore, $T'$, and thus $T$ contains a jewel.
\end{proof}

Below, we obtain a class of trees which attain the lower bound.
For a lobster, a limb
of the form $K_{1,2}$ is called a \textit{pure limb}. A lobster with
only pure limbs and no whiskers is called a \textit{pure lobster}.

\begin{corollary}\label{purelobster}
If $G$ is a pure lobster with spine length $\ell$ and
$t$ pure limbs  attached to the spine, then $\reg(S/J_G) = \ell+t$.
\end{corollary}

\begin{proof} 
Since in a pure lobster, only vertices that have degree $3$ or more
are in the spine, it can not contain the Jewel graph as a subgraph.
Therefore by Theorem \ref{tree-regularity} and Theorem \ref{minreg},
we have $\reg(S/J_G) = \iv(G) + 1 = \ell + t$.
\end{proof}

\noindent
\begin{minipage}{\linewidth}
  \begin{minipage}{0.73\linewidth}
In Theorem \ref{glue}, it was shown that the regularity of the graph
obtained by gluing two graphs at a free vertex is sum of the
regularities of these two graphs. Naturally, one tends to ask what
happens to the regularity if we glue more graphs at a free vertex. We
partially answer this question in the next theorem.
Let $\mathcal{G}(m,n,w)$ be the family of graphs obtained by identifying
a free vertex each of $K_{1,r_1}, \ldots, K_{1,r_m}$, where $r_i \geq 3$,
$n$ cliques on at least three vertices and $w$ whiskers. 
\end{minipage}
\begin{minipage}{0.25\linewidth}
  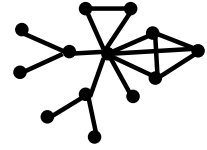
\begin{figure}[H]
\scalebox{1.5} 
{
\begin{pspicture}(0,-0.65)(1.7193422,0.65)
\psdots[dotsize=0.12,dotangle=-89.59424](0.8392708,0.16964523)
\psdots[dotsize=0.12,dotangle=-89.59424](1.237986,0.35247347)
\psdots[dotsize=0.12,dotangle=-89.59424](1.2608182,-0.047374874)
\psdots[dotsize=0.12,dotangle=-89.59424](1.639109,0.1953102)
\psline[linewidth=0.04cm](1.2579855,0.3526151)(0.8392708,0.16964523)
\psline[linewidth=0.04cm](1.2808177,-0.04723324)(0.81927127,0.1695036)
\psline[linewidth=0.04cm](1.2579855,0.3526151)(1.6591085,0.19545184)
\psline[linewidth=0.04cm](1.6591085,0.19545184)(1.2808177,-0.04723324)
\psline[linewidth=0.04cm](1.6389674,0.21530971)(0.8792697,0.1699285)
\psline[linewidth=0.04cm](1.2579855,0.3526151)(1.2808177,-0.04723324)
\psline[linewidth=0.04cm](0.8591286,0.18978637)(0.45928028,0.16695413)
\psline[linewidth=0.04cm](0.45913863,0.18695363)(0.057873942,0.36411637)
\psline[linewidth=0.04cm](0.43928078,0.1668125)(0.040565543,-0.016015723)
\psdots[dotsize=0.12,dotangle=-89.59424](0.49913764,0.1872369)
\psdots[dotsize=0.12,dotangle=-89.59424](0.0777318,0.38425753)
\psdots[dotsize=0.12,dotangle=-89.59424](0.060423404,0.0041254126)
\psline[linewidth=0.04cm](0.82204473,0.21)(1.0620447,0.57)
\psline[linewidth=0.04cm](0.80204475,0.21)(0.6420447,0.57)
\psline[linewidth=0.04cm](0.6420447,0.57)(1.0620447,0.57)
\psline[linewidth=0.04cm](0.80204475,0.13)(0.68204474,-0.21)
\psline[linewidth=0.04cm](0.66204476,-0.19)(0.30204475,-0.41)
\psline[linewidth=0.04cm](0.66204476,-0.19)(0.74204475,-0.57)
\psdots[dotsize=0.12](0.6420447,-0.19)
\psdots[dotsize=0.12](0.30204475,-0.39)
\psdots[dotsize=0.12](0.72204477,-0.57)
\psline[linewidth=0.04cm](0.8420447,0.17)(1.0620447,-0.23)
\psdots[dotsize=0.12](1.0620447,-0.21)
\psdots[dotsize=0.12](0.6420447,0.57)
\psdots[dotsize=0.12](1.0420448,0.57)
\end{pspicture} 
}
\caption*{ {\small $G \in \mathcal{G}(2,2,1)$}}
\end{figure}
\end{minipage}
\end{minipage}
\begin{theorem}\label{gmnw}
If $G \in \mathcal{G}(m,n,w), ~ n \geq 2,$ then
$\reg(S/J_G) = n+2m.$
\end{theorem}

\begin{proof}
We prove by induction on $m$. Let $m = 0$. If $w = 0$, then the result
follows from Kiani-Madani. Suppose $w \geq 1$. Let $v$ denote the
vertex which is common to all the cliques and whiskers. Let $G'$ be
the clique on $V(G)$, $G''$ be the graph induced on $V(G) \setminus
\{v\}$ and $H$ be the graph $G'\setminus \{v\}$. Then $\reg(S/J_{G'}) =
\reg(S/J_H) = 1$. Since $G''$ is a collection of $n$ disjoint cliques
and $w$ isolated vertices, $\reg(S/J_{G''}) = n$. Therefore, the
assertion follows from the exact sequence:
\begin{eqnarray}\label{eh-seq}
0 \to S/J_G \to S/J_{G'} \oplus S/J_{G''} \to S/J_H \to 0.
\end{eqnarray}

Now assume that $m \geq 1$.
Let $\{u\}$ be a leaf vertex in $G$ and
$\{u,v\} \in
E(G)$. Let $G'$ be the graph obtained by adding necessary edges to $G$
so that $N[v]$ is a clique. Let $G''$ be the induced subgraph of $G$
on $V(G) \setminus \{v\}$. Let $H$ be the induced subgraph of $G'$ on
$V(G') \setminus \{v\}$. Therefore, we have the exact sequence
(\ref{eh-seq}).
Note that $G', H \in \mathcal{G}(m-1, n+1,w)$ and $G'' \in
\mathcal{G}(m-1, n,w)$. Therefore, by induction hypothesis
$\reg(S/J_{G'}) = \reg(S/J_H) = n + 2m - 1$ and $\reg(S/J_{G''}) = n +
2m - 2$. Therefore, from the short exact sequence, we get $\reg(S/J_G)
\leq n + 2m$. Note that $G$ contains $n$ vertex disjoint edges and $m$
vertex disjoint paths length $2$ as an induced subgraph. Therefore,
$\reg(S/J_G) \geq n + 2m$. 
\end{proof}

We now consider another subclass of block graphs and obtain
an improved upper bound on the regularity of binomial
edge ideals of those graphs.

\begin{theorem}\label{clique-whisk}
Let $G$ be the union $P\cup C_1 \cup \cdots \cup C_r \cup L_1 \cup
\cdots \cup L_t \cup e_1 \cup \cdots \cup e_w$ where $P$ is a longest
induced path on the vertices $\{v_0,\ldots,v_\ell\},$ $C_1, \ldots,
C_r$ are maximal cliques on at least three vertices, $L_1, \ldots,
L_t$ be limbs and $e_1, \ldots, e_w$ are whiskers such that $e_i \cap
\{v_0,v_\ell\} = \emptyset$ and 
\begin{enumerate}
  \item For all $A, B \in \{C_1, \ldots, C_r, L_1, \ldots, L_t, e_1,
	\ldots, e_w\}$ with $A \neq B$, 
	\begin{enumerate}
	  \item $A \cap B \subset P$ and $|A \cap B| \leq 1$; 
	  \item $|A \cap P| = 1$.
	\end{enumerate}
\end{enumerate}
Then 
$\reg(S/J_G) \leq \ell+2t + r.$
\end{theorem}

\begin{proof}
Without loss of generality, we assume that there are no degree $2$
vertices in $\{v_0,\ldots, v_\ell\}$. Further, we may assume that
there is a clique, say $C_i$, such that $C_i \cap P = \{v_\ell\}.$
If not, then $v_\ell$ is a leaf in $G$.  
Attach a clique $C'$ to $\{v_\ell\}$, to get a graph $G_1$ having
$\reg(S/J_{G_1}) = \reg(S/J_G) + 1$ (by gluing theorem).

We prove the assertion by induction on $t$. Let $t = 0$.
We argue this case by induction on $\ell$.
Suppose $\ell = 0$. Since $e_i \cap
\{v_0,v_\ell\} = \emptyset$, $w = 0$. Hence $G \in
\mathcal{G}(0,r,0)$. Hence the result follows from Theorem \ref{gmnw}.

Let $\ell = 1$. Suppose $v_0 \in C_1$. Let $G', ~G''$ and $H$ be the
graphs obtained by applying EHH-process on $G$ with respect to $v_0$.
Then $G''$ is a block graph with exactly $r$-cliques, $G'$ and $H$ are
block graphs with at most $r$-cliques. Therefore, it follows from the
exact sequence (\ref{ehh-es}) and \cite[Theorem 3.5]{mk-arxiv-13} that
\begin{eqnarray*}
  \reg(S/J_G) & \leq & \max\{\reg(S/J_{G'}), \reg(S/J_{G''}),
  \reg(S/J_H)+1\} \leq r+1.
\end{eqnarray*}

Now, suppose $\ell \geq 2$. 
Without loss of generality, assume that 
$C_1 \cap \{v_{\ell-1}, v_\ell\} \neq \emptyset$. 
Let $G',~ G''$  and $H$ be the graphs obtained by applying
EHH-process on $G$ with respect to $v_{\ell-1}$. 
Then $G'$ is the union of path $P'$
of length at most $\ell-1$ containing $\{v_0,\ldots, v_{\ell-2}\}$ and
cliques, $\{C_1', C_2,\ldots,C_r\}$ and a subset of whiskers
$\{e_1,\ldots, e_w\}$. Hence by induction, $\reg(S/J_{G'}) \leq
\ell-1+r$. Similarly $\reg(S/J_H) \leq \ell-1+r$ and $\reg(S/J_{G''})
\leq \ell-1+r$. Therefore, it follows from the exact sequence
(\ref{ehh-es}) that $\reg(S/J_G) \leq \ell+r$.

Let us assume that $t \geq 1$. Since there are no degree $2$ vertices
in $G$, each limb $L_i = K_{1,\mu_i}$ for some $\mu_i \geq 3$.
Consider the limb $L_t$. Suppose $L_t \cap P = \{v\}$ and $N_{L_t}(v)
= \{u\}$. Let $G', ~G''$ and $H$ be the
graphs obtained by applying EHH-process on $G$ with respect to $u$.
Then $G'$ and $H$ both have 
spines of length $\ell$, $r+1$ cliques, $t-1$ limbs  and $w$ whiskers.
Also, $G''$ has spine of length $\ell$, $r$ cliques, $t-1$ limbs and
$w$ whiskers. Therefore, by induction, $\reg(S/J_{G'}), \reg(S/J_H)
\leq \ell+(r+1)+2(t-1)$, and $\reg(S/J_{G''}) \leq \ell+r+2(t-1)$.
Hence, from the exact sequence (\ref{ehh-es}), it follows that $\reg(S/J_G)
\leq \ell+(r+1) + 2(t-1) + 1 = \ell+r+2t$ as required.
\end{proof}

Note that the graphs $G$ considered in Theorem \ref{clique-whisk} are
block graphs, and hence by Theorem 3.5 of \cite{mk-arxiv-13}, one has
$\reg(S/J_G) \leq c(G),$ where $c(G)$ is the number of cliques in $G$.
In the case where the $L_i = K_{1,r_i}$ for $r_i \geq 3$ and $w > 0$,
the bound given above is much smaller to the Madani-Kiani bound.

As a consequence of the above theorem, we generalize a result of
Chaudhry et al. to obtain an upper bound on the regularity of lobster
graphs.
\begin{corollary}\label{cor:lobsternowhisker}
If $G$ is a lobster with spine $P$ of length $\ell$ and $t$ limbs 
$P$, then $\reg(S/J_G) \leq \ell+2t$. 
\end{corollary}
\begin{proof}
  Take $r = 0$ in Theorem \ref{clique-whisk}.
\end{proof}

\vskip 2mm

\begin{example}
This is an example of a lobster which attains the upper
bound given in Theorem \ref{cor:lobsternowhisker}.
\vskip 2mm
\noindent
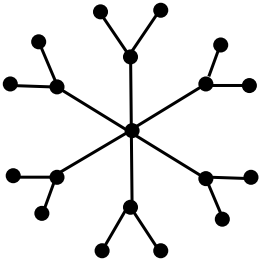
\begin{minipage}{\linewidth}
  \centering
\begin{minipage}{0.4\linewidth}
\begin{figure}[H]
\begin{pspicture}(0,-2.0060937)(4.1846876,2.0060937)
\psdots[dotsize=0.2](2.0803125,-0.00109375)
\psdots[dotsize=0.2](2.0603125,0.9789063)
\psdots[dotsize=0.2](2.0603125,-1.0210937)
\psdots[dotsize=0.2](1.0803125,-0.62109375)
\psdots[dotsize=0.2](1.0803125,0.57890624)
\psdots[dotsize=0.2](3.0603125,0.61890626)
\psdots[dotsize=0.2](3.0603125,-0.64109373)
\psline[linewidth=0.04cm](2.0603125,0.99890625)(2.0803125,-1.0010937)
\psline[linewidth=0.04cm](2.0803125,0.01890625)(3.0603125,0.61890626)
\psline[linewidth=0.04cm](1.0803125,-0.5810937)(2.0603125,-0.00109375)
\psline[linewidth=0.04cm](1.0803125,0.57890624)(3.0603125,-0.64109373)
\psdots[dotsize=0.2](1.6803125,-1.6010938)
\psdots[dotsize=0.2](2.4603126,-1.6010938)
\psline[linewidth=0.04cm](2.0403125,-0.98109376)(1.6803125,-1.6010938)
\psline[linewidth=0.04cm](2.0603125,-1.0210937)(2.4403124,-1.6010938)
\psdots[dotsize=0.2](1.6603125,1.5789063)
\psdots[dotsize=0.2](2.4603126,1.5989063)
\psline[linewidth=0.04cm](1.6603125,1.5589062)(2.0403125,0.99890625)
\psline[linewidth=0.04cm](2.0403125,0.99890625)(2.4603126,1.6189063)
\psline[linewidth=0.04cm](1.0603125,0.6589062)(0.8403125,1.1789062)
\psline[linewidth=0.04cm](1.0403125,0.57890624)(0.4603125,0.5989063)
\psdots[dotsize=0.2](0.8403125,1.1789062)
\psdots[dotsize=0.2](0.4603125,0.61890626)
\psline[linewidth=0.04cm](3.3003125,-1.2010938)(3.0803125,-0.68109375)
\psline[linewidth=0.04cm](3.6803124,-0.64109373)(3.1003125,-0.62109375)
\psdots[dotsize=0.2](3.6603124,-0.62109375)
\psdots[dotsize=0.2](3.2803125,-1.1810937)
\psline[linewidth=0.04cm](3.1003125,0.7189062)(3.2603126,1.1589062)
\psline[linewidth=0.04cm](3.1203125,0.5989063)(3.6603124,0.5989063)
\psline[linewidth=0.04cm](0.9003125,-1.0810938)(1.0603125,-0.64109373)
\psline[linewidth=0.04cm](1.0403125,-0.62109375)(0.5203125,-0.62109375)
\psdots[dotsize=0.2](0.8803125,-1.1010938)
\psdots[dotsize=0.2](0.5003125,-0.60109377)
\psdots[dotsize=0.2](3.6403124,0.5989063)
\psdots[dotsize=0.2](3.2603126,1.1389062)
\end{pspicture} 
\caption{$G$}
\label{fig:doublejewel}
\end{figure}
\end{minipage}
\begin{minipage}{0.5\linewidth}
This graph $G$ has many different longest induced paths. Fixing any
one of them, one can see that $G$ has spine length $\ell=4, ~t=4$
limbs attached to the spine and $2$ whiskers. It can be shown that
$$\reg(S/J_G) = 12 = \ell+2t.$$ 
\end{minipage}
\end{minipage}
\end{example}

\begin{corollary}
Let $G$ be a lobster with spine $P$ of length $\ell$, $t$ limbs and
$r$ whiskers. Then $\ell + t \leq \reg(S/J_G) \leq \ell+2t$.
\end{corollary}
\begin{proof}
The upper bound is proved in Corollary \ref{cor:lobsternowhisker}. To prove the
lower bound, note that $G$ has a subgraph $G'$ with spine $P$, $t$
pure limbs and without any whiskers as an induced subgraph. By
Corollary
\ref{purelobster}, $\reg(S/J_{G'}) = \ell+t$ as required.
\end{proof}
\vskip 2mm
\noindent
From Theorem \ref{minreg}, it is clear that the presence of Jewel
graph as a subgraph plays crucial role in determining the regularity
of a tree. It can be seen that $\reg(S/J_T) \geq \iv(T) + j$, where
$T$ contains $j$ vertex disjoint copies of the Jewel graph. We believe
that understanding the regularity behaviour of collection of Jewels
that share vertices and/or edges can lead to a precise estimation of
regularity of trees.

Recently, Herzog and Rinaldo has generalized Theorem
\ref{tree-regularity} to certain block graphs.

\vskip 2mm
\noindent
\textbf{Acknowledgment:} We thank Nathann Cohen for
setting up SAGE and giving us initial lessons in programming. 
We have
extensively used computer algebra software SAGE, \cite{sage}, and
Macaulay2, \cite{M2}, for our computations. Thanks are also due to Jinu
Mary Jameson who provided us with a lot of computational materials.
This research is partly funded by I.C.S.R. Exploratory Project
Grant, MAT/1415/831/RFER/AVJA, of I.I.T. Madras and Extra Mural
Research project by Sciences and Engineering Research Board,
Government of India Grant, EMR/2016/001883. We would also
like to thank the anonymous referee for a meticulous reading and
making several suggestions which improved the exposition.

\bibliographystyle{plain}
\bibliography{binombib}
\end{document}